\documentclass[11pt]{amsart}
\usepackage{amsmath}
\usepackage[active]{srcltx}
\usepackage{t1enc}
\usepackage[latin2]{inputenc}
\usepackage{verbatim}
\usepackage{amsmath,amsfonts,amssymb,amsthm}
\usepackage[mathcal]{eucal}
\usepackage{enumerate}
\usepackage[centertags]{amsmath}
\usepackage{graphics}
\usepackage[active]{srcltx}

\newtheorem{theorem}{Theorem}
\newtheorem{lemma}{Lemma}
\newtheorem*{TK}{Theorem T}

\numberwithin{equation}{subsection}

\begin{document}
\author{Ushangi Goginava and Larry Gogoladze}
\title[Convergence in Norm]{Convergence in Norm of logarithmic means of
Multiple Fourier series}
\address{U. Goginava, Department of Mathematics, Faculty of Exact and
Natural Sciences, Ivane Javakhishvili Tbilisi State University, Chavchavadze
str. 1, Tbilisi 0128, Georgia}
\email{zazagoginava@gmail.com}
\address{L. Gogoladze, Department of Mathematics, Faculty of Exact and
Natural Sciences, Ivane Javakhishvili Tbilisi State University, Chavchavadze
str. 1, Tbilisi 0128, Georgia}
\email{lgogoladze1@hotmail.com}
\maketitle

\begin{abstract}
The maximal Orlicz space such that the mixed logarithmic means of multiple
Fourier series for the functions from this space converge in $L_{1}$-norm is
found.
\end{abstract}

\footnotetext{%
2010 Mathematics Subject Classification 42A24 .
\par
The research of U. Goginava was supported by project
Shota Rustaveli National Science Foundation grant no.31/48 (Operators in some function spaces and their applications in
Fourier analysis)}

Let $\mathbb{T}^{d}:=[-\pi ,\pi )^{d}$ denote a cube in the $d$-dimensional
Euclidean space $\mathbb{R}^{d}$. The elements of $\mathbb{R}^{d}$ are
denoted by $\mathbf{x:=}\left( x_{1},...,x_{d}\right) $.

Let $D:=\left\{ 1,2,...,d\right\} ,B:=\left\{ l_{1},l_{2},...,l_{r}\right\}
,1\leq r\leq d,B\subset M,l_{k}<l_{k+1},k=1,2,...,r-1,B^{\prime
}:=D\backslash B$. For any $\mathbf{x=}\left( x_{1},...,x_{d}\right) $ and
any $B\subset D$ , denote $\mathbf{x}_{B}:=\left(
x_{l_{1}},x_{l_{2}},...,x_{l_{r}}\right) \in \mathbb{R}^{r}$. We assume that
$|B|$ is the number of elements of $B$. If $B\neq \varnothing $, then for
any natural numbers $n$ we suppose that $n\left( B\right) :=\left(
n,n,...,n\right) \in \mathbb{R}^{|B|}.$ The notation $a\lesssim b$ in the
paper stands for $a\leq cb$, where $c$ is an absolute constant.

In the sequel we shall identify the symbols%
\begin{equation*}
\sum\limits_{i_{B}=0_{B}}^{n_{B}}\text{ \ and }\sum%
\limits_{i_{l_{1}}=0}^{n_{l_{1}}}\cdots
\sum\limits_{i_{l_{r}}=0}^{n_{l_{r}}},d\mathbf{t}_{B}\text{ and }%
dt_{l_{1}}\cdots dt_{l_{r}}.
\end{equation*}

We denote by $L_{p}\left( \mathbb{T}^{d}\right) $ the class of all
measurable functions $f$ that are $2\pi $-variable with respect to all
variable and satisfy%
\begin{equation*}
\left\Vert f\right\Vert _{p}:=\left( \int\limits_{\mathbb{T}%
^{d}}|f|^{p}\right) ^{1/p}<\infty .
\end{equation*}

Let $f\in L_{1}\left( \mathbb{T}^{d}\right) .$ The Fourier series of $f$
with respect to the trigonometric system is the series
\begin{equation*}
S\left[ f\right] :=\sum_{n_{1},...,n_{d}=-\infty }^{+\infty }\widehat{f}%
\left( n_{1},...,n_{d}\right) e^{i\left( n_{1}x_{1}+\cdots
+n_{d}x_{d}\right) },
\end{equation*}%
where
\begin{equation*}
\widehat{f}\left( n_{1},...,n_{d}\right) :=\frac{1}{\left( 2\pi \right) ^{d}}%
\int\limits_{\mathbb{T}^{d}}f(x_{1},...,x_{d})e^{-i\left( n_{1}x_{1}+\cdots
+n_{d}x_{d}\right) }dx_{1}\cdots dx_{d}
\end{equation*}%
are the Fourier coefficients of the function $f$. The rectangular partial
sums are defined as follows:
\begin{equation*}
S_{N_{D}}(f;\mathbf{x}):=\sum_{n_{D}=-N_{D}}^{N_{D}}\widehat{f}\left(
n_{1},...,n_{d}\right) e^{i\left( n_{1}x_{1}+\cdots +n_{d}x_{d}\right) }.
\end{equation*}

In the literature, there is known the notion of the Riesz's logarithmic
means of a Fourier series. The $n$-th Riesz logarithmic mean of the Fourier
series of the integrable function $f$ is defined by
\begin{equation*}
\frac{1}{l_{n}}\sum_{k=0}^{n}\frac{S_{k}(f)}{k+1},l_{n}:=\sum_{k=0}^{n}\frac{%
1}{k+1},
\end{equation*}%
where $S_{k}(f)$ is the partial sum of its Fourier series. This Riesz's
logarithmic means with respect to the trigonometric system has been studied
by a lot of authors. We mention for instance the papers of Szász, and Yabuta
\cite{sza, ya}. This mean with respect to the Walsh, Vilenkin system is
discussed by Simon, and Gát \cite{Sim, gat}.

Let $\left\{ q_{k}:k\geq 0\right\} $ be a sequence of nonnegative numbers.
The Nörlund means for the Fourier series of $f$ are defined by
\begin{equation*}
\frac{1}{\sum_{k=0}^{n}q_{k}}\sum_{k=0}^{n}q_{k}S_{n-k}(f).
\end{equation*}%
If $q_{k}=\frac{1}{k+1}$, then we get the (Nörlund) logarithmic means:
\begin{equation*}
\frac{1}{l_{n}}\sum_{k=0}^{n}\frac{S_{n-k}(f)}{k+1}.
\end{equation*}
Although, it is a kind of \textquotedblleft reverse\textquotedblright\
Riesz's logarithmic means. In \cite{gg} we proved some convergence and
divergence properties of the logarithmic means of Walsh-Fourier series of
functions in the class of continuous functions, and in the Lebesgue space $L$%
. \

The Nörlund logarithmic and Reisz logarithmic means of multiple Fourier
series are defined by%
\begin{equation*}
L_{n_{D}}\left( f;\mathbf{x}\right) :=\frac{1}{\prod\limits_{i\in D}l_{i}}%
\sum\limits_{i_{D}=0_{D}}^{n_{D}}\frac{S_{n_{D}-i_{D}}\left( f;\mathbf{x}%
\right) }{\prod\limits_{j\in D}\left( i_{j}+1\right) },
\end{equation*}%
\begin{equation*}
R_{n_{D}}\left( f;\mathbf{x}\right) :=\frac{1}{\prod\limits_{i\in D}l_{i}}%
\sum\limits_{i_{D}=0_{D}}^{n_{D}}\frac{S_{i_{D}}\left( f;\mathbf{x}\right) }{%
\prod\limits_{j\in D}\left( i_{j}+1\right) }.
\end{equation*}

It is evident that
\begin{equation*}
L_{n_{D}}\left( f;\mathbf{x}\right) =\frac{1}{\pi ^{d}}\int\limits_{\mathbb{T%
}^{d}}f\left( \mathbf{t}\right) F_{n_{D}}\left( \mathbf{x}-\mathbf{t}\right)
d\mathbf{t}
\end{equation*}%
and%
\begin{equation*}
R_{n_{D}}\left( f;\mathbf{x}\right) =\frac{1}{\pi ^{d}}\int\limits_{\mathbb{T%
}^{d}}f\left( \mathbf{t}\right) G_{n_{D}}\left( \mathbf{x}-\mathbf{t}\right)
d\mathbf{t,}
\end{equation*}%
where
\begin{equation*}
F_{n_{D}}\left( \mathbf{x}\right) :=\prod\limits_{j\in D}F_{n_{j}}\left(
x_{j}\right) ,G_{n_{D}}\left( \mathbf{x}\right) :=\prod\limits_{j\in
D}G_{n_{j}}\left( x_{j}\right) ,
\end{equation*}%
\begin{equation*}
F_{n}\left( u\right) :=\frac{1}{l_{n}}\sum\limits_{i=0}^{n}\frac{%
D_{n-i}\left( u\right) }{i+1},G_{n}\left( u\right) :=\frac{1}{l_{n}}%
\sum\limits_{i=0}^{n}\frac{D_{i}\left( u\right) }{i+1}.
\end{equation*}

Let $B\subset D$. Then the mixed logarithmic means of multiple Fourier
series are defined by
\begin{equation*}
\left( L_{n_{B}}\circ R_{n_{B^{\prime }}}\right) \left( f;\mathbf{x}\right)
:=\frac{1}{\prod\limits_{i\in D}l_{i}}\sum\limits_{i_{D}=0_{D}}^{n_{D}}\frac{%
S_{n_{B}-i_{B},i_{B^{\prime }}}\left( f;\mathbf{x}\right) }{%
\prod\limits_{j\in D}\left( i_{j}+1\right) }.
\end{equation*}%
It is easy to show that%
\begin{equation*}
\left( L_{n_{B}}\circ R_{n_{B^{\prime }}}\right) \left( f;\mathbf{x}\right) =%
\frac{1}{\pi ^{d}}\int\limits_{\mathbb{T}^{d}}f\left( \mathbf{t}\right)
F_{n_{B}}\left( \mathbf{x}_{B}-\mathbf{t}_{B}\right) G_{n_{B^{\prime
}}}\left( \mathbf{x}_{B^{\prime }}-\mathbf{t}_{B^{\prime }}\right) d\mathbf{t%
}
\end{equation*}

We denote by $L_{0}(\mathbb{T}^{d})$ the Lebesgue space of functions that
are measurable and finite almost everywhere on $\mathbb{T}^{d}$. mes$\left(
A\right) $ is the Lebesgue measure of the set $A\subset \mathbb{T}^{d}$.

Let $L_{Q}=L_{Q}(\mathbb{T}^{d})$ be the Orlicz space \cite{KR} generated by
Young function $Q$, i.e. $Q$ is convex continuous even function such that $%
Q(0)=0$ and

\begin{equation*}
\lim\limits_{u\rightarrow +\infty }\frac{Q\left( u\right) }{u}=+\infty
,\,\,\,\,\lim\limits_{u\rightarrow 0}\frac{Q\left( u\right) }{u}=0.
\end{equation*}

This space is endowed with the norm
\begin{equation*}
\Vert f\Vert _{L_{Q}(\mathbb{T}^{d})}=\inf \{k>0:\int\limits_{\mathbb{T}%
^{d}}Q(\left\vert f\right\vert /k)\leq 1\}.
\end{equation*}

In particular, if $Q(u)=u\log ^{r}(1+u)$ ,$r=1,2,...,\quad u>0$, then the
corresponding space will be denoted by $L\log ^{r}L(\mathbb{T}^{d})$.

The rectangular partial sums of double Fourier series $S_{n,m}\left(
f;x,y\right) $ of the function $f\in L_{p}\left( \mathbb{T}^{2}\right) ,%
\mathbb{T}:=[-\pi ,\pi ),1<p<\infty $ converge in $L_{p}$ norm to the
function $f$, as $n\rightarrow \infty $ \cite{Zh}. In the case $L_{1}\left(
\mathbb{T}^{2}\right) $ this result does not hold . But for $f\in
L_{1}\left( \mathbb{T}\right) $, the operator $S_{n}\left( f;x\right) $ are
of weak type (1,1) \cite{Zy}. This estimate implies convergence of $%
S_{n}\left( f;x\right) $ in measure on $\mathbb{T}$ to the function $f\in
L_{1}\left( \mathbb{T}\right) $. However, for double Fourier series this
result does not hold \cite{Ge, Kon}. \ Moreover, it is proved that
quadratical partial sums $S_{n,n}\left( f;x,y\right) $ of double Fourier
series do not converge in two-dimensional measure on $\mathbb{T}^{2}$ \ even
for functions from Orlicz spaces wider than Orlicz space $L\log L\left(
\mathbb{T}^{2}\right) $. On the other hand, it is well-known that if the
function $f\in L\log L\left( \mathbb{T}^{2}\right) $, then rectangular
partial sums $S_{n,m}\left( f;x,y\right) $ converge in measure on $\mathbb{T}%
^{2}$.

Classical regular summation methods often improve the convergence of Fourier
series. For instance, the Fejér means of the double Fourier series of the
function $f\in L_{1}\left( \mathbb{T}^{2}\right) $ converge in $L_{1}\left(
\mathbb{T}^{2}\right) $ norm to the function $f$ \cite{Zh}. These means
present the particular case of the Nörlund means.

It is well know that the method of Nörlund logarithmic means of double
Fourier series, is weaker than the Cesáro method of any positive order. In
\cite{Tk} Tkebuchava proved, that these means of double Fourier series in
general do not converge in two-dimensional measure on $\mathbb{T}^{d}$ even
for functions from Orlicz spaces wider than Orlicz space $L\log
^{d-1}L\left( \mathbb{T}^{d}\right) $. In particular, the following result
is true.

\begin{TK}
\label{tkebuchava}Let $L_{Q}\left( \mathbb{T}^{d}\right) $ be an Orlicz
space, such that%
\begin{equation*}
L_{Q}\left( \mathbb{T}^{d}\right) \nsubseteq L\log ^{d-1}L\left( \mathbb{T}%
^{d}\right) .
\end{equation*}%
Then the set of the function from the Orlicz space $L_{Q}\left( \mathbb{T}%
^{d}\right) $ with Nörlund logarithmic means of rectangular partial sums of $d$-dimensional
Fourier series, convergent in measure on $\mathbb{T}^{d}$, is of first Baire
category in $L_{Q}\left( \mathbb{T}^{d}\right) $.
\end{TK}

In \cite{GG} we considered the strong logarithmic means of rectangular
partial sums double Fourier series%
\begin{equation*}
\sigma _{n,m}\left( f;x,y\right) :=\frac{1}{l_{n}l_{m}}\sum\limits_{i=0}^{n}%
\sum\limits_{j=0}^{m}\frac{\left\vert S_{i,j}\left( f;x,y\right) \right\vert
}{\left( n-i+1\right) \left( m-j+1\right) }
\end{equation*}%
and prove that these means are acting from space $L\log L\left( \mathbb{T}%
^{2}\right) $ into space $L_{p}\left( \mathbb{T}^{2}\right) ,0<p<1$. This
fact implies the convergence of strong logarithmic means of rectangular
partial sums of double Fourier series in measure on $\mathbb{T}^{2}$ to the
function $f\in L\log L\left( \mathbb{T}^{2}\right) $ . Uniting these results
with statement from \cite{Tk1} we obtain, that the rectangular partial sums
of double Fourier series converge in measure for all functions from Orlicz
space if and only if their strong Nörlund logarithmic means  converge in measure. Thus, not all classic regular
summation methods can improve the convergence in measure of double Fourier
series.

The results for summability of logarithmic means of Walsh-Fourier series can
be found in \cite{GGT2, GGNSMH,gg, GogJAT,sza,ya}.

In this paper we consider the mixed logarithmic means $\left( L_{n_{B}}\circ
R_{n_{B^{\prime }}}\right) \left( f\right) $ of rectangular partial sums
multiple Fourier series and prove that these means are acting from space $%
L\log ^{|B|}L\left( \mathbb{T}^{d}\right) $ into space $L_{1}\left( \mathbb{T%
}^{d}\right) $ (see Theorem \ref{logest} ). This fact implies the
convergence of mixed logarithmic means of rectangular partial sums of
multiple Fourier series converge in $L_{1}$-norm. We also prove sharpness of
this result (see Theorem \ref{normdiv}). In particular, the following are true.

\begin{theorem}
\label{logest}Let $B\subset D$ and $f\in L\log ^{|B|}L\left( \mathbb{T}%
^{d}\right) $. Then%
\begin{equation*}
\left\Vert \left( L_{n_{B}}\circ R_{n_{B^{\prime }}}\right) \left( f\right)
\right\Vert _{L_{1}\left( \mathbb{T}^{d}\right) }\lesssim 1+\left\Vert
\left\vert f\right\vert \log ^{|B|}\left\vert f\right\vert \right\Vert
_{L_{1}\left( \mathbb{T}^{d}\right) }.
\end{equation*}
\end{theorem}

\begin{theorem}
\label{normconv}

Let $B\subset D$ and $f\in L\log ^{|B|}L\left( \mathbb{T}^{d}\right) $. Then
\begin{equation*}
\left\Vert \left( L_{n_{B}}\circ R_{n_{B^{\prime }}}\right) \left( f\right)
-f\right\Vert _{L_{1}\left( \mathbb{T}^{d}\right) }\rightarrow 0\text{ \ as
\ }n_{i}\rightarrow \infty ,i\in D\text{;}
\end{equation*}
\end{theorem}

\begin{theorem}
\label{normdiv} Let $L_{Q}\left( \mathbb{T}^{d}\right) $ be an Orlicz space,
such that%
\begin{equation*}
L_{Q}\left( \mathbb{T}^{d}\right) \nsubseteqq L\log ^{|B|}L\left( \mathbb{T}%
^{d}\right) .
\end{equation*}%
Then\newline
a)%
\begin{equation*}
\sup\limits_{n}\left\Vert \left( L_{n\left( B\right) }\circ R_{n\left(
B^{\prime }\right) }\right)  \right\Vert _{L_{Q}\left(
\mathbb{T}^{d}\right) \rightarrow L_{1}\left( \mathbb{T}^{d}\right) }=\infty
;
\end{equation*}%
\newline
\newline
b) there exists a function $f\in L_{Q}\left( \mathbb{T}^{d}\right) $ such
that $\left( L_{n\left( B\right) }\circ R_{n\left( B^{\prime }\right)
}\right) \left( f\right) $ does not converge to $f$ in $L_{1}\left( \mathbb{T%
}^{d}\right) $-norm.
\end{theorem}

Thus, the space $L\log ^{|B|}L\left( \mathbb{T}^{d}\right) $ is maximal
Orlicz space such that for each function $f$ from this space the means $%
\left( L_{n\left( B\right) }\circ R_{n\left( B^{\prime }\right) }\right)
\left( f\right) $ converge to $f$ in $L_{1}\left( \mathbb{T}^{d}\right) $%
-norm.

\begin{proof}[Proof of Theorem \protect\ref{logest}]
We apply the following particular case of the Marcinkiewicz interpolation
theorem \cite{E}. Let $T:L_{1}\left( T^{1}\right) \rightarrow L_{0}\left(
T^{1}\right) $ be a quasilinear operator of weak type $(1,1)$ and of type $%
\left( \alpha ,\alpha \right) $ for some $1<\alpha <\infty $ at the same
time, i. e.

a)%
\begin{eqnarray}
&&\text{mes}\left\{ x\in \mathbb{T}^{1}:\left\vert T\left( f,x\right)
\right\vert >y\right\}  \label{T1} \\
&\lesssim &\frac{1}{y}\int\limits_{\mathbb{T}^{1}}\left\vert f\left(
x\right) \right\vert dx;\,\,\,\,\forall f\in L^{1}\left( \mathbb{T}%
^{1}\right) \,\,\forall y>0;  \notag
\end{eqnarray}

b)
\begin{equation}
\,\,\,\,\,\left\Vert Tf\right\Vert _{L^{\alpha }\left( \mathbb{T}^{1}\right)
}\lesssim \,\,\left\Vert f\right\Vert _{L^{\alpha }\left( \mathbb{T}%
^{1}\right) },\,\,\,\forall f\in L^{\alpha }\left( T^{1}\right) .  \label{T2}
\end{equation}%
Then%
\begin{eqnarray}
&&\int\limits_{\mathbb{T}^{1}}\left\vert T\left( f,x\right) \right\vert \ln
^{\beta }\left\vert T\left( f,x\right) \right\vert dx\,  \label{est} \\
&\lesssim &\int\limits_{\mathbb{T}^{1}}\left\vert f\left( x\right)
\right\vert \ln ^{\beta +1}\left\vert f\left( x\right) \right\vert
dx+1\,,\,\,\,\forall \beta \geq 0.  \notag
\end{eqnarray}

In \cite{GG} it proved that for any $f\in L_{1}\left( T^{1}\right) $ the
operator $f\ast F_{n}$ has weak type (1,1) , i. e.%
\begin{eqnarray}
&&\text{mes}\left\{ x\in \mathbb{T}^{1}:\left\vert f\ast F_{n}\right\vert
>y\right\}  \label{weak} \\
&\lesssim &\frac{1}{y}\int\limits_{\mathbb{T}^{1}}\left\vert f\left(
x\right) \right\vert dx;\,\,\,\,\forall f\in L^{1}\left( \mathbb{T}%
^{1}\right) \,\,\forall y>0.  \notag
\end{eqnarray}

On the other hand, it is easy to prove that the operator $f\ast G_{n}$ has
type (1,1), i.e.%
\begin{equation}
\left\Vert f\ast G_{n}\right\Vert _{L_{1}\left( \mathbb{T}^{1}\right)
}\lesssim \left\Vert f\right\Vert _{L_{1}\left( \mathbb{T}^{1}\right) }.
\label{strong}
\end{equation}

From (\ref{T1})-(\ref{strong}) we have ($B^{\prime }:=\left\{
s_{1},s_{2},...,s_{r^{\prime }}\right\} $, $B:=\left\{ l_1,...,l_r \right\}$ )%
\begin{eqnarray*}
&&\left\Vert \left( L_{n_{B}}\circ R_{n_{B^{\prime }}}\right) \left(
f\right) \right\Vert _{L_{1}\left( \mathbb{T}^{d}\right) } \\
&=&\left\Vert \left( R_{n_{s_{1}}}\circ \cdots \circ R_{n_{s_{r^{\prime
}}}}\circ L_{n_{l_{1}}}\circ \cdots \circ L_{n_{l_{r}}}\right) \left(
f\right) \right\Vert _{L_{1}\left( \mathbb{T}^{d}\right) } \\
&\lesssim &\cdots \lesssim \left\Vert \left( L_{n_{l_{1}}}\circ \cdots \circ
L_{n_{l_{r}}}\right) \left( f\right) \right\Vert _{L_{1}\left( \mathbb{T}%
^{d}\right) } \\
&\lesssim &1+\left\Vert \left\vert L_{n_{l_{2}}}\circ \cdots \circ
L_{n_{l_{r}}}\left( f\right) \right\vert \log \left\vert L_{n_{2}}\circ
\cdots \circ L_{n_{l_{r}}}\left( f\right) \right\vert \right\Vert
_{L_{1}\left( \mathbb{T}^{d}\right) } \\
&\lesssim &\cdots \lesssim 1+\left\Vert \left\vert L_{n_{l_{r}}}\left(
f\right) \right\vert \log ^{r-1}\left\vert L_{n_{l_{r}}}\left( f\right)
\right\vert \right\Vert _{L_{1}\left( \mathbb{T}^{d}\right) } \\
&\lesssim &1+\left\Vert \left\vert f\right\vert \log ^{r}\left\vert
f\right\vert \right\Vert _{L_{1}\left( \mathbb{T}^{d}\right) }.
\end{eqnarray*}

Theorem \ref{logest} is proved.
\end{proof}

The validity of Theorem \ref{normconv} follows immediately from Theorem \ref%
{logest}.

\begin{proof}[Proof of Theorem \protect\ref{normdiv}]
a) Set%
\begin{equation*}
\alpha _{mn}:=\frac{\pi \left( 12m+1\right) }{6\left( 2^{2n}+1/2\right) }%
,\beta _{mn}:=\frac{\pi \left( 12m+5\right) }{6\left( 2^{2n}+1/2\right) }%
,\gamma _{n}:=\frac{\pi }{6\left( 2^{2n}+1/2\right) },
\end{equation*}%
\begin{equation*}
J_{n}:=\bigcup\limits_{m=1}^{2^{n-1}}\left[ \alpha _{mn}+\gamma _{n},\beta
_{mn}-\gamma _{n}\right] .
\end{equation*}%
In order to prove theorem we need the following lemma proved in \cite{GT}.

\begin{lemma}
\label{GT}Let $0\leq z\leq \gamma _{n}$ and $x\in J_{n}$. Then%
\begin{equation*}
F_{2^{2n}}\left( x-z\right) \gtrsim \frac{1}{x}.
\end{equation*}
\end{lemma}

Let
\begin{equation*}
Q\left( 2^{2n|B|}\right) \gtrsim 2^{2n|B|}\text{ \ for \ }n>n_{0}.
\end{equation*}%
By virtue of estimate (\cite{KR}, Ch. 2)%
\begin{equation*}
\left\Vert f\right\Vert _{L_{Q}}\leq 1+\left\Vert Q\left( |f|\right)
\right\Vert _{L_{1}}.
\end{equation*}%
We can write%
\begin{eqnarray}
&&\left\Vert L_{2^{2n}\left( B\right) }\circ R_{2^{2n}\left( B^{\prime
}\right) }\left( \frac{\mathbf{1}_{\left[ 0,\gamma _{n}\right] ^{|B|}}}{%
\left( 2\gamma _{n}\right) ^{|B|}}\right) \right\Vert _{L_{1}\left( \mathbb{T%
}^{d}\right) }  \label{1} \\
&\leq &\left\Vert L_{2^{2n}\left( B\right) }\circ R_{2^{2n}\left( B^{\prime
}\right) }\right\Vert _{L_{Q}\left( \mathbb{T}^{d}\right) \rightarrow
L_{1}\left( \mathbb{T}^{d}\right) }\left\Vert \frac{\mathbf{1}_{\left[
0,\gamma _{n}\right] ^{|B|}}}{\left( 2\gamma _{n}\right) ^{|B|}}\right\Vert
_{L_{Q}\left( \mathbb{T}^{d}\right) }  \notag \\
&\leq &\left\Vert L_{2^{2n}\left( B\right) }\circ R_{2^{2n}\left( B^{\prime
}\right) }\right\Vert _{L_{Q}\left( \mathbb{T}^{d}\right) \rightarrow
L_{1}\left( \mathbb{T}^{d}\right) }  \notag \\
&&\times \left( 1+\left\Vert Q\left( \frac{\mathbf{1}_{\left[ 0,\gamma _{n}%
\right] ^{|B|}}}{\left( 2\gamma _{n}\right) ^{|B|}}\right) \right\Vert
_{L_{1}\left( \mathbb{T}^{d}\right) }\right)  \notag \\
&\lesssim &\left\Vert L_{2^{2n}\left( B\right) }\circ R_{2^{2n}\left(
B^{\prime }\right) }\right\Vert _{L_{Q}\left( \mathbb{T}^{d}\right)
\rightarrow L_{1}\left( \mathbb{T}^{d}\right) }\left( 1+\gamma
_{n}^{|B|}Q\left( \frac{1}{\left( 2\gamma _{n}\right) ^{|B|}}\right) \right)
\notag \\
&\lesssim &\left\Vert L_{2^{2n}\left( B\right) }\circ R_{2^{2n}\left(
B^{\prime }\right) }\right\Vert _{L_{Q}\left( \mathbb{T}^{d}\right)
\rightarrow L_{1}\left( \mathbb{T}^{d}\right) }\frac{Q\left(
2^{2n|B|}\right) }{2^{2n|B|}}.  \notag
\end{eqnarray}

From Lemma \ref{GT} we get%
\begin{eqnarray*}
&&L_{2^{2n}\left( B\right) }\circ R_{2^{2n}\left( B^{\prime }\right) }\left(
\frac{\mathbf{1}_{\left[ 0,\gamma _{n}\right] ^{|B|}}}{\left( 2\gamma
_{n}\right) ^{|B|}};\mathbf{x}\right)  \\
&=&\frac{1}{\left( 2\gamma _{n}\right) ^{|B|}}\frac{1}{\pi ^{d}}\int\limits_{%
\left[ 0,\gamma _{n}\right] ^{|B|}}\prod\limits_{j\in B}F_{2^{2n}}\left(
x_{j}-z_{j}\right) d\mathbf{z}_{B} \\
&&\times \int\limits_{\mathbb{T}^{|B^{\prime }|}}\prod\limits_{i\in
B^{\prime }}G_{2^{2n}}\left( x_{i}-z_{i}\right) d\mathbf{z}_{B^{\prime }} \\
&=&\frac{1}{\left( 2\gamma _{n}\right) ^{|B|}}\prod\limits_{j\in B}\frac{1}{%
\pi }\int\limits_{\left[ 0,\gamma _{n}\right] }F_{2^{2n}}\left(
x_{j}-z_{j}\right) dz_{j} \\
&&\times \prod\limits_{i\in B^{\prime }}\frac{1}{\pi}\int\limits_{\mathbb{T}%
}G_{2^{2n}}\left( z_{i}\right) dz_{i} \\
&=&\frac{1}{\left( 2\gamma _{n}\right) ^{|B|}}\prod\limits_{j\in B}\frac{1}{%
\pi }\int\limits_{\left[ 0,\gamma _{n}\right] }F_{2^{2n}}\left(
x_{j}-z_{j}\right) dz_{j} \\
&\gtrsim &\frac{1}{\left( 2\gamma _{n}\right) ^{|B|}}\prod\limits_{j\in B}%
\frac{\gamma _{n}}{x_{j}},x_{j}\in J_{n},j\in B.
\end{eqnarray*}%
Consequently,%
\begin{eqnarray}
&&\left\Vert L_{2^{2n}\left( B\right) }\circ R_{2^{2n}\left( B^{\prime
}\right) }\left( \frac{\mathbf{1}_{\left[ 0,\gamma _{n}\right] ^{|B|}}}{%
\left( 2\gamma _{n}\right) ^{|B|}}\right) \right\Vert _{L_{1}\left( \mathbb{T%
}^{d}\right) }  \label{low} \\
&\gtrsim &\prod\limits_{j\in B}\int\limits_{J_{n}}\frac{dx_{j}}{x_{j}}%
\gtrsim cn^{|B|}.  \notag
\end{eqnarray}

Combining (\ref{1}) and (\ref{low}) we obtain%
\begin{equation}
\left\Vert L_{2^{2n}\left( B\right) }\circ R_{2^{2n}\left( B^{\prime
}\right) }\right\Vert _{L_{Q}\left( \mathbb{T}^{d}\right) \rightarrow
L_{1}\left( \mathbb{T}^{d}\right) }\gtrsim \frac{2^{2n|B|}n^{|B|}}{Q\left(
2^{2n|B|}\right) }.  \label{main}
\end{equation}

The fact that
\begin{equation*}
L_{Q}\left( \mathbb{T}^{d}\right) \nsubseteqq L\log ^{|B|}L\left( \mathbb{T}%
^{d}\right)
\end{equation*}%
is equivalent to the condition%
\begin{equation*}
\lim\limits_{u\rightarrow \infty }\sup \frac{u\log ^{|B|}u}{Q\left( u\right)
}=\infty .
\end{equation*}%
Thus there exists $\left\{ u_{k}:k\geq 1\right\} $ such that%
\begin{equation*}
\lim\limits_{k\rightarrow \infty }\frac{u_{k}\log ^{|B|}u_{k}}{Q\left(
u_{k}\right) }=\infty ,u_{k+1}>u_{k},k=1,2,...,
\end{equation*}%
and a monotonically increasing sequence of positive integers $\left\{
r_{k}:k\geq 1\right\} $ such that%
\begin{equation*}
2^{2|B|r_{k}}\leq u_{k}<2^{2|B|\left( r_{k}+1\right) }.
\end{equation*}%
Then we have%
\begin{equation*}
\frac{2^{2r_{k}|B|}r_{k}^{|B|}}{Q\left( 2^{2r_{k}|B|}\right) }\gtrsim \frac{%
u_{k}\log ^{|B|}u_{k}}{Q\left( u_{k}\right) }\rightarrow \infty .
\end{equation*}%
Thus from (\ref{main}) we conclude that%
\begin{equation*}
\sup\limits_{n}\left\Vert \left( L_{n\left( B\right) }\circ R_{n\left(
B^{\prime }\right) }\right)  \right\Vert _{L_{Q}\left(
\mathbb{T}^{d}\right) \rightarrow L_{1}\left( \mathbb{T}^{d}\right) }=\infty
.
\end{equation*}%
This complete the proof of Theorem \ref{normdiv} a). Part b) follows
immediately from part a).
\end{proof}

\end{document}